\documentclass{degruyter-proceedings}
\usepackage{latexsym,amsfonts,amsmath,graphics}
\usepackage{epsfig}
\usepackage{enumitem}
\usepackage{amssymb,mathrsfs}

\newcommand{\rd}{\,\mathrm{d}}

\newcommand{\bsx}{\boldsymbol{x}}
\newcommand{\bsy}{\boldsymbol{y}}

\newcommand{\bsl}{\boldsymbol{l}}
\newcommand{\bsk}{\boldsymbol{k}}
\newcommand{\bst}{\boldsymbol{t}}

\newcommand{\bszero}{\boldsymbol{0}}

\newcommand{\RR}{\mathbb{R}}

\newcommand{\LL}{{\cal L}}

\newcommand{\FF}{\mathbb{F}}
\newcommand{\NN}{\mathbb{N}}

\newcommand{\wal}{{\rm wal}}
\newcommand{\cP}{\mathcal{P}}
\newcommand{\cS}{\mathcal{S}}

\allowdisplaybreaks

\title{Explicit constructions of point sets and sequences with low discrepancy}
\headlinetitle{Constructions of point sets and sequences with low discrepancy}

\lastnameone{Dick}
\firstnameone{Josef}
\nameshortone{J.~Dick}
\addressone{School of Mathematics and Statistics, The University of New South Wales, Sydney, NSW 2052}
\countryone{Australia}
\emailone{josef.dick@unsw.edu.au}

\lastnametwo{Pillichshammer}
\firstnametwo{Friedrich}
\nameshorttwo{F.~Pillichshammer}
\addresstwo{Department of Financial Mathematics,
Johannes Kepler University Linz, Altenbergerstra{\ss}e 69, 4040 Linz}
\countrytwo{Austria}
\emailtwo{friedrich.pillichshammer@jku.at}

\abstract{In this article we survey recent results on the explicit construction of finite point sets and infinite sequences with optimal order of $\LL_q$ discrepancy. In 1954 Roth proved a lower bound for the $\LL_2$ discrepancy of finite point sets in the unit cube of arbitrary dimension. Later various authors extended Roth's result to lower bounds also for the $\LL_q$ discrepancy and for infinite sequences. While it was known already from the early 1980s on that Roth's lower bound is best possible in the order of magnitude, it was a longstanding open question to find explicit constructions of point sets and sequences with optimal order of $\LL_2$ discrepancy. This problem was solved by Chen and Skriganov in 2002 for finite point sets and recently by the authors of this article for infinite sequences. These constructions can also be extended to give optimal order of the $\LL_q$ discrepancy of finite point sets for $q \in (1,\infty)$. The main aim of this article is to give an overview of these constructions 
and related results.
}



\researchsupported{}

\firstpage{1}

\begin{document}
\section{Introduction}

We consider equidistribution properties of sequences in the
$s$-dimensional unit-cube $[0,1)^s$ measured by their $\LL_q$ discrepancy (see \cite{BC, DP10, DT97,
kuinie,mat}). For a finite set $\cP_{N,s} =\{\bsx_0,\ldots
,\bsx_{N-1}\}$ of points in the $s$-dimensional unit-cube $[0,1)^s$
the local discrepancy function is defined as
$$\Delta_{\cP_{N,s}}(\bst)=\frac{A_N([\bszero,\bst),\cP_{N,s})}{N}- t_1 t_2 \cdots t_s,$$ where $\bst = (t_1,t_2,\ldots, t_s) \in [0,1]^s$ and
$A_N([\bszero,\bst), \cP_{N,s})$ denotes the number of indices $n$ with
$\bsx_n \in [0,t_1)\times \dots \times [0,t_s) =: [\bszero, \bst)$. The local discrepancy function measures the difference of the portion of points in an axis parallel box containing the origin and the volume of this box. Hence it is a measure of the irregularity of distribution of a point set in $[0,1)^s$.

\begin{definition}
Let $q \in [1,\infty]$. The {\it $\LL_q$ discrepancy} of $\cP_{N,s}$ is defined as the $\LL_q$-norm of the local discrepancy function
\begin{eqnarray}\label{Lq-definition}
\LL_{q,N}(\cP_{N,s})=\| \Delta_{\cP_{N,s}}\|_{\LL_q}= \left(\int_{[0,1]^{s}} |\Delta_{\cP_{N,s}}(\bst)|^q \rd\bst\right)^{1/q}
\end{eqnarray}
with the obvious modifications for $q=\infty$. For an infinite sequence $\cS_s$ in $[0,1)^s$ the $\LL_q$ discrepancy $\LL_{q,N}(\cS_s)$ is the $\LL_q$ discrepancy of the point set consisting of the first $N$ elements of $\cS_s$.
\end{definition}

It is well known that a sequence is uniformly distributed modulo one if and only if its $\LL_q$ discrepancy tends to zero for growing $N$. Furthermore, the $\LL_q$ discrepancy can also be linked to the integration error of a quasi-Monte Carlo rule, see, e.g. \cite{DP10,Nie73,SloWo} for the error in the worst-case setting and \cite{Woz} for the average case setting.

One of the questions on irregularities of distribution is concerned with the precise order of convergence of the smallest possible values of the $\LL_q$ discrepancy as $N$ goes to infinity. While this problem is completely solved for $q \in (1,\infty)$, the case $q \in \{1,\infty\}$ appears to be much more difficult. In particular, for $q=\infty$ and $s \ge 3$ the exact asymptotic order of the smallest possible value of the $\LL_{\infty}$ discrepancy is still unknown (for $s=2$ it is known to be $(\log N)/N$). There are many people who conjecture that the sharp order of magnitude for the smallest possible value of the $\LL_{\infty}$ discrepancy of $N$-element point sets in $[0,1)^s$ is $(\log N)^{s-1}/N$. But there are also other opinions such as, for example, $(\log N)^{s/2}/N$, see \cite{BilLac}. Although there is some recent remarkable progress, which is surveyed in \cite{BilLac}, the exact determination of the sharp order of magnitude for the smallest possible value of $\LL_{\infty}$ discrepancy remains
unknown.

In this paper we only deal with the case $q \in (1,\infty)$ and survey the development of the problem beginning with Roth's seminal lower bound on the $\LL_2$ discrepancy from 1954 (Section~\ref{sec_lb}) to the recent constructions of point sets and sequences with optimal order of $\LL_2$ discrepancy (Section~\ref{sec_ub}), on which we put our main focus. In detail, in Section~\ref{sec_lb} we discuss Roth's lower bound for the $\LL_2$ discrepancy of finite point sets and its extensions to $\LL_q$ discrepancy and to infinite sequences. Section~\ref{sec_ub} deals with existence results for point sets and sequences whose orders of magnitude of the $\LL_2$ discrepancy match the lower bounds. Sections~\ref{sec_DC} introduces digital nets and sequences which provide the basic ideas for the explicit constructions. Section~\ref{sec_walsh} introduces Walsh functions, which is the main analytical tool to obtain discrepancy bounds. This section also motivates the ideas for the latter constructions. Section~\ref{sec_CS} 
introduces the explicit constructions of point sets with optimal $\LL_2$ discrepancy by Chen and Skriganov~\cite{CS02}, whereas Section~\ref{sec_DP} introduces the explicit constructions of point sets and sequences with optimal $\LL_2$ discrepancy of \cite{DP13}. Section~\ref{sec_lq} describes the extensions of the latter results to the $\LL_q$ discrepancy. Section~\ref{sec_orlicz} briefly discusses recent extensions of the $\LL_q$ discrepancy results to exponential Orlicz norms.

We close this introduction with some notation that is used throughout this survey: Let $\NN$ be the set of positive integers and let $\NN_0=\NN \cup \{0\}$.  Furthermore let $b$ be a prime number and let $\FF_b$ be the finite field of order $b$ identified with the set $\{0,1,\ldots,b-1\}$ equipped with arithmetic operations modulo $b$. For functions $g,f: \NN \rightarrow \RR$, $f>0$, we write  $g(N) \ll f(N)$ if there is a constant $c > 0$ such that $g(N) \le c f(N)$ for all $N \in \NN$. If the implied constant $c$ depends on some parameters, say, for example, $s$, we denote this by $g(N) \ll_s f(N)$.

\section{Lower bounds}\label{sec_lb}

In 1954 Roth~\cite{Roth} proved a celebrated general lower bound on the $\LL_2$ discrepancy of finite point sets.

\begin{theorem}[Roth, 1954]\label{thm1}
For every dimension $s \in \NN$ there exists a real $c_s >0$ with the following property: For any integer $N \ge 2$ and any $N$-element point set $\cP_{N,s}$ in $[0,1)^s$ we have $$\LL_{2,N}(\cP_{N,s}) \ge c_s \frac{(\log N)^{\frac{s-1}{2}}}{N}.$$
\end{theorem}

Roth's original proof can be found in \cite{Roth}. Further proofs are presented in \cite{BC,DP10,DT97,HM,kuinie,mat}. According to \cite{HM} the quantity $c_s$ can be chosen as $$c_s= \frac{7}{27 \cdot 2^{2s-1} (\log 2)^{(s-1)/2} \sqrt{(s-1)!}}.$$ From the monotonicity of the $\LL_q$ norm it is evident that Roth's lower bound also holds for the $\LL_q$ discrepancy for any $q \ge 2$. Furthermore, it was shown by Schmidt~\cite{schX} that also for any $q > 1$ there exists a real $c_{s,q}>0$ with the property that
\begin{equation}\label{extsch}
\LL_{q,N}(\cP_{N,s}) \ge c_{s,q} \frac{(\log N)^{\frac{s-1}{2}}}{N}
\end{equation}
for any $N$-element point set $\cP_{N,s}$ in $[0,1)^s$.

In 1986 Proinov~\cite{pro86} extended Roth's lower bound to infinite sequences.

\begin{theorem}[Proinov, 1986]\label{thm2}
For every dimension $s \in \NN$ and any $q>1$ there exists a real $\alpha_{s,q} >0$ with the following property:  For any infinite sequence $\cS_s$ in $[0,1)^s$ we have $$\LL_{q,N}(\cS_{s}) \ge \alpha_{s,q} \frac{(\log N)^{\frac{s}{2}}}{N} \ \ \ \ \mbox{ for infinitely many }\ N \in \NN.$$
\end{theorem}

Since the paper including Proinov's proof is not widely available we present the proof of Theorem~\ref{thm2} according to the original paper. To this end we require the following well known lemma which goes back to Roth~\cite{Roth}.

\begin{lemma}\label{le1}
For $s\in \NN$ let $\cS_s=(\bsy_k)_{k\ge 0}$, where $\bsy_k=(y_{1,k},\ldots,y_{s,k})$ for $k \in \NN_0$, be an arbitrary sequence in the $s$-dimensional unit cube with $\LL_q$ discrepancy $\LL_{q,N}(\cS_s)$ for $1 \le q \le \infty$. Then for every $N\in \NN$ there exists an $n \in \{1,2,\ldots,N\}$ such that $$n \LL_{q,n}(\cS_s) \ge N \LL_{q,N}(\cP_{N,s+1})-1$$ where $\cP_{N,s+1}$ is the finite point set in $[0,1)^{s+1}$ consisting of the points $$\bsx_k=(y_{1,k},\ldots,y_{s,k},k/N) \ \ \mbox{ for }\ k=0,1,\ldots ,N-1.$$
\end{lemma}

\begin{proof}
We present the proof for finite $q$. For $q=\infty$ the proof is similar and can be found in \cite[Lemma~3.54]{DP10} or \cite[Lemma~3.7]{niesiam}.
Choose $n \in \{1,2,\ldots,N\}$ such that $$n \LL_{q,n}(\cS_s) =\max_{k=1,2,\ldots,N} k \LL_{q,k}(\cS_s).$$

Consider a sub-interval of the $s+1$-dimensional unit cube of the form $E=\prod_{i=1}^{s+1}[0,t_i)$ with $\bst=(t_1,t_2,\ldots,t_{s+1}) \in [0,1)^{s+1}$ and put $m=m(t_{s+1}):=\lceil N t_{s+1}\rceil$. Then a point $\bsx_k$, $k=0,1,\ldots, N-1$, belongs to $E$, if and only if $\bsy_k \in \prod_{i=1}^s[0,t_i)$ and $0 \le k < N t_{s+1}$. Denoting $E'=\prod_{i=1}^s[0,t_i)$ we have $A_N(E,\cP_{N,s+1})=A_m(E',\cS_s)$ and therefore
\begin{align*}
N \Delta_{\cP_{N,s+1}}(\bst) = & A_N(E,\cP_{N,s+1}) -N  t_1 t_2\cdots t_{s+1}\\
= & A_m(E',\cS_s) - m t_1 t_2 \cdots t_s + m t_1 t_2 \cdots t_s-N  t_1 t_2\cdots t_{s+1}\\
= & m \Delta_{\cS_s}(\bst')+ t_1 t_2 \cdots t_s (m-N t_{s+1}),
\end{align*}
where $\bst'=(t_1,\ldots,t_s)$.
From the definition of $m$ it is clear that $|m-N t_{s+1}| \le 1$ and hence we obtain $$N |\Delta_{\cP_{N,s+1}}(\bst)| \le m |\Delta_{\cS_s}(\bst')| +1.$$ Furthermore, for every $t_{s+1} \in [0,1)$ we have $m \LL_{m,q}(\cS_s) \le n \LL_{n,q}(\cS_s)$. Hence we obtain
\begin{align*}
  N \LL_{q,N} (\cP_{N,s+1}) \le &  \|m \Delta_{\cS_s} +1\|_{\LL_q} \le \|m \Delta_{\cS_s}\|_{\LL_q} +1\\
= & \left(\int_0^1 \int_{[0,1]^s} | m \Delta_{\cS_s} (\bst') |^q \rd \bst' \rd t_{s+1}\right)^{1/q} +1\\
= & \left(\int_0^1 (m \LL_{m,q}(\cS_s))^q \rd t_{s+1}\right)^{1/q}+1\\
\le & n \LL_{n,q}(\cS_s) +1
\end{align*}
which completes the proof of the lemma.
\end{proof}

Now we can give the proof of Theorem~\ref{thm2} according to Proinov~\cite{pro86}.

\begin{proof}[Proof of Theorem~\ref{thm2}]
We use the notation from Lemma~\ref{le1}. For the $\LL_q$ discrepancy of the finite point set $\cP_{N,s+1}$ in $[0,1)^{s+1}$ we obtain from Theorem~\ref{thm1} and Schmidt's extension \eqref{extsch} that $$N \LL_{q,N}(\cP_{N,s+1}) \ge c_{s+1,q} (\log N)^{\frac{s}{2}}$$ for some real $c_{s+1,q}>0$ which is independent of $N$. Let $\alpha_{s,q} \in (0,c_{s+1,q})$ and let $N \in \NN$ be large enough such that $c_{s+1,q} (\log N)^{\frac{s}{2}} -1 \ge \alpha_{s,q} (\log N)^{\frac{s}{2}}$. According to Lemma~\ref{le1} there exists an $n \in \{1,2,\ldots,N\}$ such that
\begin{equation}\label{eq1}
n \LL_{2,n}(\cS_s) \ge N \LL_{2,N}(\cP_{N,s+1})-1 \ge c_{s+1,q} (\log N)^{\frac{s}{2}} -1 \ge \alpha_{s,q} (\log n)^{\frac{s}{2}}.
\end{equation}
Thus we have shown that for every large enough $N \in \NN$ there exists an $n \in \{1,2,\ldots,N\}$ such that
\begin{equation}\label{eq2}
n \LL_{2,n}(\cS_s) \ge \alpha_{s,q} (\log n)^{\frac{s}{2}}.
\end{equation}
It remains to show that \eqref{eq2} holds for infinitely many $n \in \NN$. Assume on the contrary that \eqref{eq2} holds for finitely many $n \in \NN$ only and let $m$ be that largest integer with this property. Then choose $N \in \NN$ large enough such that $$c_{s+1,q} (\log N)^{\frac{s}{2}} -1 \ge \alpha_{s,q} (\log N)^{\frac{s}{2}} > \max_{k=1,2,\ldots,m} k \LL_{q,k}(\cS_s).$$ For this $N$ we can find an $n \in \{1,2,\ldots,N\}$ for which \eqref{eq1} and \eqref{eq2} hold true. However, \eqref{eq1} implies that $n > m$ which leads to a contradiction since $m$ is the largest integer such that \eqref{eq2} is true. Thus we have shown that \eqref{eq2} holds for infinitely many $n \in \NN$ and this completes the proof.
\end{proof}

\section{Upper bounds}\label{sec_ub}

In 1956 Davenport \cite{dav} proved that Theorem~\ref{thm1} is best possible in dimension 2. He considered the $N=2 M$ points $(\{\pm n \alpha\},n/M)$ for $n=1,2,\ldots,M$ and showed that if $\alpha$ is an irrational number having a continued fraction expansion with bounded partial quotients then the $\LL_2$ discrepancy of the collection $\cP_{N,2}^{{\rm sym}}(\alpha)$ of these points satisfies $$\LL_{2,N}(\cP_{N,2}^{{\rm sym}}(\alpha)) \ll_{\alpha} \frac{\sqrt{\log N}}{N}$$ where the implied constant only depends on $\alpha$. Nowadays there exist several variants of such ``symmetrized'' point sets having optimal order of $\LL_2$ discrepancy in dimension 2, see, for example, \cite{lp} or \cite{pro1988a}. A nice discussion of the topic, which is often referred to as {\it Davenport's reflection principle} can be found in \cite{CS2000}. Recently, Bilyk \cite{Bil} proved that unsymmetrized versions of such point sets, i.e., point sets of the form $\cP_{N,2}(\alpha)=\{(\{n \alpha\},n/N)\ : \, n=0,1,\ldots,N-1\}$
satisfy $$\LL_{2,N}(\cP_{N,2}(\alpha)) \ll_{\alpha} \frac{\sqrt{\log N}}{N}$$ if and only if the bounded partial quotients of $\alpha=[a_0;a_1,a_2,\ldots]$ satisfy $$\left|\sum_{k=0}^n(-1)^k a_k\right| \ll_\alpha \sqrt{n}.$$ Further examples of two-dimensional finite point sets with optimal order of $\LL_2$ discrepancy which are based on scrambled digital nets can be found in \cite{FauPi09a,FauPi09,FauPiPri11,FauPiPriSch09,HZ,KriPi2006}.

Concerning higher dimensions, in \cite{roth2} Roth proved that the bound from Theorem~\ref{thm1} is best possible in dimension 3 and finally Roth \cite{Roth4} and Frolov \cite{Frolov} proved that Theorem~\ref{thm1} is best possible in any dimension. In \cite{chen1980} Chen showed that the $\LL_q$ discrepancy bound \eqref{extsch} is best possible in the order of magnitude in $N$ for any $q>1$, i.e., for every $N,s \in \NN$, $N \ge 2$, there exists an $N$-element point set $\cP_{N,s}$ in $[0,1)^s$ such that $$\LL_{q,N}(\cP_{N,s}) \ll_{s,q} \frac{(\log N)^{\frac{s-1}{2}}}{N},$$ where the implied constant only depends on $s$ and $q$, but not on $N$. See also \cite{BC} for more information. Further existence results for point sets with optimal order of $\LL_q$ discrepancy can be found in \cite{CS3,DP05b,Skr12}. However, all these results for dimension 3 and higher are only existence results obtained by averaging arguments and it remained a long standing open question in discrepancy theory to find explicit
constructions of finite point sets with optimal order of $\LL_q$ discrepancy in the sense of Roth's lower bound. The breakthrough in this direction was achieved by Chen and Skriganov \cite{CS02}, who provided a complete solution to this problem. They gave, for the first time, for every integer $N \ge 2$ and every dimension $s \in \NN$, explicit constructions of finite $N$-element point sets in $[0,1)^s$ whose $\LL_2$ discrepancy achieves an order of convergence of $(\log N)^{(s-1)/2}/N$. Their construction, which will be explained in Section~\ref{sec_CS}, uses a finite field $\FF_b$ of order $b$ with $b \ge 2 s^2$. The result in \cite{CS02} was extended to the $\LL_q$ discrepancy for $1 \le q < \infty$ by Skriganov~\cite{Skr}. \\

For one-dimensional infinite sequences there are some examples with optimal order of $\LL_2$ discrepancy in the sense of Proinov's lower bound from Theorem~\ref{thm2}. For example the symmetrized sequence $(\{\alpha\},\{-\alpha\},\{2 \alpha\},\{-2 \alpha\},\{3 \alpha\},\{-3\alpha\},\ldots)$, where $\alpha$ is an irrational number having a continued fraction expansion with bounded partial quotients, which goes back to Davenport (see \cite[Theorem~1.75]{DT97}). Other examples are based on symmetrized van der Corput sequences, see, for example, \cite{chafa,fau90,lp,pro1988a}. However, in spite of the construction of Chen and Skriganov of finite point sets with optimal order of magnitude of the $\LL_2$ discrepancy there was still no explicit construction of an {\it infinite} sequence in $[0,1)^s$ with optimal order of $\LL_2$ discrepancy in the sense of Theorem~\ref{thm2} for dimensions $s \ge 2$. This problem was recently solved in \cite{DP13} where the authors provide  explicit constructions of {\it infinite} 
sequences in $[0,1)
^s$ for which the first $N \ge 2$ points achieve a  $\LL_2$ discrepancy of order $(\log N)^ {s/2}/N$ for arbitrary $s \in \NN$. This construction is based on higher order digital sequences over the finite field $\FF_2$ and will be explained in Section~\ref{sec_DP}. This construction also yields a new construction of finite point sets with optimal order of $\LL_2$ discrepancy.\\

Since the finite constructions of Chen and Skriganov as well as the constructions of infinite sequences are based on the digital construction method, we will explain this construction scheme in the following section.

\section{Digital nets and sequences}\label{sec_DC}

The concepts of digital nets and sequences were introduced by Niederreiter~\cite{nie87} in 1987 and are nowadays among the most powerful methods for the construction of low discrepancy point sets and sequences. These constructions are based on linear algebra over $\FF_b$. A detailed overview of this topic is given in the books \cite{DP10,niesiam}.

First we recall the definition of digital nets according to Niederreiter which we present here in a slightly more general form. For $m,p \in \NN$ with $p \ge m$ let $C_1,\ldots, C_s \in \FF_b^{p \times m}$ be $p \times m$ matrices over $\FF_b$ (originally one uses $m \times m$ matrices). For $n \in \{0,\ldots ,b^m-1\}$ with $b$-adic expansion $n = n_0 + n_1 b + \cdots + n_{m-1} b^{m-1}$ define the $b$-ary digit vector $\vec{n}$ as $\vec{n} = (n_0, n_1, \ldots, n_{m-1})^\top \in \FF_b^m$ (the symbol $\top$ means the transpose of a vector or a matrix). Then compute
\begin{equation*}
C_j \vec{n} =:(x_{j,n,1}, x_{j,n,2},\ldots,x_{j,n,p})^\top \quad \mbox{for } j = 1,\ldots, s,
\end{equation*}
where the matrix vector product is evaluated over $\FF_b$, and put
\begin{equation*}
x_{j,n} = x_{j,n,1} b^{-1} + x_{j,n,2} b^{-2} + \cdots + x_{j,n,p} b^{-p} \in [0,1).
\end{equation*}
The $n$th point $\boldsymbol{x}_n$ of the net $\cP_{b^m,s}$ is given by $\boldsymbol{x}_n = (x_{1,n}, \ldots, x_{s,n})$. A net $\cP_{b^m,s}$ constructed this way is called a {\it digital net (over $\FF_b$) with generating matrices} $C_1,\ldots,C_s$. Note that a digital net consists of $b^m$ elements in $[0,1)^s$.

In the following we briefly describe the geometrical properties of digital nets. According to Niederreiter \cite{nie87} a $(t,m,s)$-net in base $b$ is a $b^m$-element point set $\cP_{b^m,s}$ in $[0,1)^s$ such that every interval of the form $$\prod_{j=1}^s \left[\frac{a_j}{b^{d_j}}, \frac{a_j +1}{b^{d_j}}\right)$$ where $d_1,\ldots,d_s \in \NN_0$ with $d_1+\cdots+d_s=m-t$ and $a_j \in \{0,1,\ldots,b^{d_j}-1\}$ for $j=1,\ldots,s$ contains exactly $b^t$ elements of $\cP_{b^m,s}$. The following result connects digital nets with $(t,m,s)$-nets (see \cite[Theorem~4.52]{DP10} or \cite[Theorem~4.28]{niesiam}).
\begin{lemma}[Niederreiter, 1987]\label{lem_net_dig_net}
Let $\cP_{b^m,s}$ be a digital net over $\mathbb{F}_b$ with generating matrices $C_1,\ldots, C_s$. Let $C_j = (\vec{c}_{j,1}, \vec{c}_{j,2}, \ldots \vec{c}_{j,p})^\top$, i.e., $\vec{c}_{j,k}^\top$ is the $k$th row of $C_j$. Then $\cP_{b^m,s}$ is a $(t,m,s)$-net in base $b$ if and only if for all $d_1, d_2,\ldots, d_s \in \mathbb{N}_0$ with $d_1+ \cdots + d_s = m-t$, the vectors
\begin{equation*}
\vec{c}_{1,1}, \vec{c}_{1,2}, \ldots, \vec{c}_{1,d_1}, \ldots, \vec{c}_{s,1}, \vec{c}_{s,2},\ldots, \vec{c}_{s,d_s}
\end{equation*}
are linearly independent over $\mathbb{F}_b$.
\end{lemma}
A digital net over $\FF_b$ which is a $(t,m,s)$-net in base $b$ is called a digital $(t,m,s)$-net over $\FF_b$.

We also recall the definition of digital sequences according to Niederreiter, which are infinite versions of digital nets. Let $C_1,\ldots, C_s \in \FF_b^{\mathbb{N} \times \mathbb{N}}$ be $\mathbb{N} \times \mathbb{N}$ matrices over $\FF_b$. For $C_j = (c_{j,k,\ell})_{k, \ell \in \mathbb{N}}$ we assume that for each $\ell \in \mathbb{N}$ there exists a $K(\ell) \in \mathbb{N}$ such that $c_{j,k,\ell} = 0$ for all $k > K(\ell)$. For $n \in \NN_0$ with $b$-adic expansion $n = n_0 + n_1 b + \cdots + n_{m-1} b^{m-1} \in \mathbb{N}_0$,  define the infinite $b$-adic digit vector of $n$ by $\vec{n} = (n_0, n_1, \ldots, n_{m-1}, 0, 0, \ldots )^\top \in \FF_b^{\mathbb{N}}$. Then compute
\begin{equation}\label{eq_dig_seq}
C_j \vec{n}=:(x_{j,n,1}, x_{j,n,2},\ldots)^\top \quad \mbox{for } j = 1,\ldots, s,
\end{equation}
where the matrix vector product is evaluated over $\FF_b$, and put
\begin{equation*}
x_{j,n} = x_{j,n,1} b^{-1} + x_{j,n,2} b^{-2} + \cdots \in [0,1).
\end{equation*}
Note that \eqref{eq_dig_seq} is only a finite sum, since it can be written as
\begin{equation*}
 \sum_{\ell=1}^m c_{j,k,\ell} n_{\ell-1} = x_{j,n,k} \quad \mbox{for } j = 1, \ldots, s \mbox{ and } k \in \mathbb{N}.
\end{equation*}
Since $c_{j,k,\ell} = 0$ for all $k$ large enough it follows that the $x_{j,n,k}$ eventually become zero. This implies that the numbers $x_{j,n}$ have always a finite base $b$ expansion.

The $n$th point $\boldsymbol{x}_n$ of the sequence $\cS_s$ is given by $\boldsymbol{x}_n = (x_{1,n}, \ldots, x_{s,n})$. A sequence $\cS_s$ constructed this way is called a {\it digital sequence (over $\FF_b$) with generating matrices} $C_1,\ldots,C_s$. 

In the following we briefly describe the geometrical properties of digital sequences. According to Niederreiter \cite{nie87} a $(t,s)$-sequence in base $b$ is an infinite sequence $(\bsx_n)_{n \ge 0}$ in $[0,1)^s$ such that for all $m,k \in \NN_0$ the point set consisting of $$\bsx_{k b^m},\bsx_{k b^m +1}, \ldots, \bsx_{(k+1) b^m-1}$$ forms a $(t,m,s)$-net in base $b$. Let $C_1,\ldots, C_s$ be the generating matrices of a digital sequence. Let $C^{(m)}_j$ denote the left upper $m \times m$ matrix of $C_j$. Using Lemma~\ref{lem_net_dig_net} it can be shown that a digital sequence $\cS_s$ is a $(t,s)$-sequence in base $b$ if and only if the matrices $C^{(m)}_1,\ldots, C_s^{(m)} \in \mathbb{F}_b^{m \times m}$ generate a digital $(t,m,s)$-net over $\FF_b$. A digital sequence over $\FF_b$ which is a $(t,s)$-sequence in base $b$ is then called a digital $(t,s)$-sequence over $\FF_b$. See \cite[Theorem~4.84]{DP10} or \cite[Theorem~4.36]{niesiam} for more details.

Explicit constructions of suitable generating matrices $C_1,\ldots, C_s$ over $\FF_b$ were obtained by Sobol'~\cite{sob67}, Faure~\cite{fau82}, Niederreiter~\cite{nie87,niesiam}, Niederreiter-Xing~\cite{NX96} and others (see \cite[Chapter~8]{DP10} for an overview). However, these generating matrices are particularly designed to have low $\LL_{\infty}$ discrepancy (which is of order $(\log N)^s/N$ for infinite sequences and $(\log N)^{s-1}/N$ for finite point sets) and it is not known whether they achieve the optimal order for $\LL_q$ discrepancy for finite $q$.

\section{Walsh series expansion of discrepancy function}\label{sec_walsh}

All the currently known constructions of point sets and sequences which achieve the optimal order of the $\LL_q$ discrepancy in dimension $s > 2$ make use of the Walsh series expansion of the discrepancy function. We describe this expansion and the necessary background in the following.

Let the real number $x \in [0,1)$ have $b$-adic expansion $x =
\frac{x_1}{b} + \frac{x_2}{b^2} + \cdots$, with $x_i \in \{0,1,\ldots ,b-1\}$ and
where infinitely many $x_i$ are different from $b-1$. For $k \in
\mathbb{N}$, $k = \kappa_1 b^{a_1-1} + \cdots +
\kappa_{\nu}b^{a_\nu-1}$, $a_1 > \cdots > a_\nu >0$ and $
\kappa_1,\ldots, \kappa_\nu \in \{1,\ldots, b-1\}$, we define the $k$th Walsh function
by
$$\mathrm{wal}_k(x) = \omega_b^{\kappa_1 x_{a_1} + \cdots + \kappa_v x_{a_v}},$$ where
$\omega_b = \exp(2\pi \mathrm{i}/b)$. For $k = 0$ we set
$\mathrm{wal}_0(x) = 1$.

In dimensions $s > 1$ we use products of the Walsh functions. Let $\bsx = (x_1, x_2,\ldots, x_s) \in [0,1]^s$ and $\bsk = (k_1, k_2, \ldots, k_s) \in \mathbb{N}_0^s$. Then we define the $\bsk$th Walsh function by
\begin{equation*}
\mathrm{wal}_{\bsk}(\bsx) = \prod_{j=1}^s \wal_{k_j}(x_j).
\end{equation*}

For a function $f:[0,1]^s \rightarrow \mathbb{R}$ we define the $\bsk$th
Walsh coefficient of $f$ by $$\widehat{f}(\bsk) = \int_{[0,1]^s} f(\bsx)
\overline{\mathrm{wal}_{\bsk}(\bsx)} \mathrm{\, d} \bsx$$ and we can form the
Walsh series $$f(\bsx) \sim \sum_{\bsk \in \mathbb{N}_0^s} \widehat{f}(\bsk)
\mathrm{wal}_{\bsk}(\bsx).$$ Note that we will not assume that we have point-wise equality in the above equation. Instead we use Parseval's identity $$\int_{[0,1]^s} |f(\bsx)|^2 \,\mathrm{d} \bsx = \sum_{\bsk \in \mathbb{N}_0^s} |\widehat{f}(\bsk)|^2,$$ which holds since the Walsh function system is complete in $\LL_2([0,1]^s)$ and does not require point-wise equality.

Note that throughout the paper Walsh functions and digital nets and sequences are
defined using the same prime number $b$.

What makes the Walsh functions so useful for analyzing the $\LL_2$ discrepancy (and the $\LL_q$ discrepancy) of digital nets is the character property. This means that for a digital net $\cP_{N,s}$ with generating matrices $C_1, C_2, \ldots, C_s \in \FF_b^{p \times m}$ we have
\begin{equation}\label{eq_char}
\frac{1}{N} \sum_{n=0}^{N-1} \wal_{\bsk}(\bsx_n) = \left\{\begin{array}{rl} 1 & \mbox{if } \bsk \in \mathcal{D}, \\ 0 & \mbox{otherwise}, \end{array} \right.
\end{equation}
where
\begin{equation*}
\mathcal{D} = \mathcal{D}(C_1,\ldots, C_s) = \{\bsk \in \mathbb{N}_0^s: C_1^\top \vec{k}_1 + \cdots + C_s^\top \vec{k}_s = \bszero\},
\end{equation*}
where for $k$ with $b$-adic expansion $\kappa_0 + \kappa_1 b + \kappa_2 b^2 + \cdots$ we write $\vec{k} = (\kappa_0, \kappa_1, \ldots, \kappa_{p-1})$.

We now consider the Walsh series expansion of the local discrepancy function $\Delta_{\cP_{N,s}}$ for digital nets $\cP_{N,s}$. We can write
\begin{equation*}
\Delta_{\cP_{N,s}}(\bst) = \frac{1}{N} \sum_{n=0}^{N-1} 1_{[\bszero,\bst)}(\bsx_n) - t_1 t_2\cdots t_s,
\end{equation*}
where $1_{[\bszero, \bst)}$ is the indicator function, i.e., $1_{[\bszero,\bst)}(\bsx_n)$ is $1$ if $\bsx_n \in [\bszero,\bst)$ and 0 otherwise. Since the Walsh series expansions of the indicator function and polynomials are well known \cite{Fine, Price} one can obtain the Walsh series representation of $\Delta_{P_{N,s}}$. Substituting this expansion into the definition of the $\LL_2$ discrepancy and using Parseval's identity then yields an expression of the form
\begin{equation}\label{eq_L2_dual}
\LL_{2,N}^2(\cP_{N,s}) = \sum_{\bsk, \bsl \in \mathcal{D} \setminus \{\bszero\}} r(\bsk, \bsl),
\end{equation}
where $r(\bsk, \bsl) = \prod_{j =1}^s r(k_j, l_j)$. We describe the structure of $r(k,l)$ in the following. To do so, assume that $k$ and $l$ have $b$-adic expansion of the form 
\begin{equation}\label{eq_kl}
k = \kappa_1 b^{a_1-1} + \lfloor \kappa_2 b^{a_2-1} \rfloor + k'' \quad \mbox{ and } \quad l = \lambda_1 b^{c_1-1} + \lfloor \lambda_2 b^{c_2-1} \rfloor + l'',
\end{equation}
where $\kappa_1, \kappa_2, \lambda_1, \lambda_2 \in \{1,2, \ldots, b-1\}$, $a_1 > a_2 \ge 0$, $c_1 > c_2 \ge 0$ and $k'' <b^{a_2-1}$ and $l'' < b^{c_2-1}$. (Note that for $a_2=0$ we have $\lfloor \kappa_2 b^{a_2-1} \rfloor = 0$, which is used if $k$ has only one nonzero digit.) Further let $k' = k-\kappa_1 b^{a_1-1}$ and $l' = l - \lambda_1 b^{c_1-1}$. The ideas for finding explicit constructions are based on the following result: A detailed analysis of the Walsh coefficients of the discrepancy function yields that the coefficients $r(k,l)$ are roughly bounded by (cf. \cite[Lemma~1]{DP13} for the case $b=2$)
\begin{equation}\label{ineq_bound_rkl}
|r(k,l)| \ll \left\{ \begin{array}{ll} b^{-2a_1} & \mbox{if } k = l, \\ b^{-\max(a_1,c_2) - \max(a_2,c_1)} & \mbox{if } k \sim l, \\ 0 & \mbox{otherwise}, \end{array} \right.
\end{equation}
where $k \sim l$ means that either $k' = l'$, $k = l'$ or $k'=l$. We write $k \not\sim l$ if $k \neq l$ and we do not have $k \sim l$. In other words $r(k,l)$ is of order $b^{-2a_1}$ if $k=l$, is of order $b^{-\max(a_1,c_2) - \max(a_2,c_1)}$ if $k$ is `very similar' to $l$, and $0$ if $k$ is `not very similar' to $l$. In other words, the coefficients $r(k,l)$ have some sparsity property.

It is well understood that (cf. \cite{CS02, CS3, DP05b})
\begin{equation}\label{bound_L2_diagonal}
\sum_{\bsk \in \mathcal{D}} r(\bsk,\bsk) \ll b^{-2m+2t} (m-t+1)^{s-1}.
\end{equation}
Thus, the digital $(t,m,s)$-net property takes care of the `diagonal' terms $r(\bsk,\bsk)$ where $\bsk = \bsl$. The difficulty in obtaining optimal bounds on the $\LL_2$ discrepancy therefore lies solely in finding constructions which also make sure that the `nondiagonal' terms $r(\bsk,\bsl)$, where $\bsk \neq \bsl$, are small. Below we describe two different strategies to find explicit constructions of generating matrices which yield digital nets and sequences with optimal order of the $\LL_2$ discrepancy. Before we do so, we introduce some metrics on $\mathbb{N}$ which will be useful for the subsequent discussion.\footnote{There is also a natural way of studying point sets and their discrepancy in algebraic terms. Using the $b$-adic expansion of $x \in [0,1)$, one can map an element $(x_1, x_2, \ldots)$ in the sequence space $\mathbb{F}_b^{\mathbb{N}}$ to the point $x = \sum_{i=1}^\infty x_i b^{-i} \in [0,1)$. This mapping is not injective since  for instance $(1,0,0,\ldots)$ and $(0, b-1, b-1, \ldots)$ 
get mapped to $1/b$. However, since the point sets we are studying (digital nets), all have a finite $b$-adic expansion, this problem never occurs in this context. Thus, instead of considering $b$-adic rationals in $[0,1)^s$ one can instead study subsets of $(\mathbb{F}_b^{\mathbb{N}})^s$ (in fact $(\mathbb{F}_b^{r})^s$ for $r$ large enough is sufficient). By \eqref{eq_char}, the set of Walsh functions forms the character theoretic dual of $\mathbb{F}_b^{\mathbb{N}}$, which can be identified with elements in $\mathbb{F}_b^{\mathbb{N}}$ for which only finitely many components are different from $0$ (which therefore in turn can be identified with the set of nonnegative integers). Equation~\eqref{eq_L2_dual} is then a functional applied to the character-theoretic dual space (which can be identified with $\mathcal{D}$) of the `point set', which can be identified with a linear subspace of $(\mathbb{F}_b^{\mathbb{N}})^s$. The algebraic point of view is advantageous when studying for instance digital nets 
constructed over 
finite fields of prime power order or certain finite rings. However, in the simple case of prime 
fields, it is sufficient to use the harmonic analysis point of view, where we consider an orthonormal basis of $\LL_2([0,1]^s)$ (the Walsh function system in our case) and study the behaviour of the series coefficients.}

\subsubsection*{Three metrics on $\mathbb{N}_0^s$}

To facilitate the discussion below we introduce three metrics on $\mathbb{N}_0^s$. Let $k \in \mathbb{N}_0$ be as above in \eqref{eq_kl}. Then we define the {\it NRT weight} (Niederreiter~\cite{nie86} and Rosenbloom-Tsfasman~\cite{RT}) by
\begin{equation*}
\mu_1(k) = \left\{\begin{array}{ll} a_1 & \mbox{if } k > 0, \\ 0 & \mbox{if } k =0. \end{array} \right.
\end{equation*}
For vectors $\bsk = (k_1, k_2, \ldots, k_s) \in \mathbb{N}_0^s$ we set
\begin{equation*}
\mu_1(\bsk) = \mu_1(k_1) + \mu_1(k_2) + \cdots + \mu_1(k_s).
\end{equation*}
The NRT weight is closely related to the $t$-value of a digital net, in fact, we have (cf. \cite[Theorem~7.8 and Corollary~7.9]{DP10})
\begin{equation*}
m-t + 1 =   \min_{\bsk \in \mathcal{D} \setminus \{\bszero\} } \mu_1(\bsk).
\end{equation*}

Next we define the {\it Hamming weight}. For $k \in \mathbb{N}_0$ the Hamming weight $\varkappa(k)$ is the number of nonzero digits in the base $b$ expansion of $k$. To be more precise, let $k \in \mathbb{N}$ have $b$-adic expansion
\begin{equation}\label{exp_k}
k = \kappa_{a_1-1} b^{a_1-1} + \kappa_{a_2-1} b^{a_2-1} + \cdots + \kappa_{a_\nu-1} b^{a_\nu-1},
\end{equation}
where $\kappa_i \in \{1,2,\ldots, b-1\}$ and $a_1 > a_2> \cdots > a_\nu > 0$. Then
\begin{equation*}
\varkappa(k) = \left\{\begin{array}{rl} \nu & \mbox{if } k \in \mathbb{N}, \\ 0 & \mbox{if } k = 0. \end{array} \right.
\end{equation*}
For vectors $\bsk  = (k_1, k_2, \ldots, k_s) \in \mathbb{N}_0^s$ we define
\begin{equation*}
\varkappa(\bsk) = \varkappa(k_1) + \varkappa(k_2) + \cdots + \varkappa(k_s).
\end{equation*}

The third metric is a generalization of the NRT weight to higher order. Using the expansion \eqref{exp_k} we set for $\alpha \in \mathbb{N}$
\begin{equation*}
\mu_\alpha(k) = \left\{\begin{array}{ll} 0 & \mbox{if } k = 0, \\ a_1 + a_2 + \cdots + a_{\min(\nu, \alpha)} & \mbox{if } k \in \mathbb{N}. \end{array} \right.
\end{equation*}
The motivation for using this metric is the fact that we can obtain a slightly weaker version of \eqref{ineq_bound_rkl} as
\begin{equation*}
|r(k,l)| \ll \left\{ \begin{array}{ll} b^{-2\mu_1(k)} & \mbox{if } k=l, \\ b^{-\max(\mu_2(k), \mu_2(l))} & \mbox{if } k \sim l, \\ 0 & \mbox{otherwise}. \end{array} \right.
\end{equation*}
Thus $\mu_2$ naturally occurs in the bound on $|r(k,l)|$.

\subsubsection*{The approach by Chen and Skriganov}

The first approach is by Chen and Skriganov~\cite{CS02} and is based on the strategy to find digital $(0,m,s)$-nets for which for all $\bsk, \bsl \in \mathcal{D}$ with $\bsk \neq \bsl$ there is at least one coordinate $j$ with $k_j \not\sim l_j$. This yields a quasi-orthogonality since then the second case in \eqref{ineq_bound_rkl} never occurs. Given this quasi-orthogonality, \eqref{eq_L2_dual} can be written as
\begin{equation}\label{CS_L2}
\LL_{2,N}^2(\cP_{N,s}) = \sum_{\bsk \in \mathcal{D}} r(\bsk,\bsk).
\end{equation}
In this case \eqref{bound_L2_diagonal} yields the result.

To achieve quasi-orthogonality we use the following observation. Let $(\kappa_0, \kappa_1, \ldots)$ be the digit vector of $k$ and let $(\lambda_1, \lambda_2, \ldots)$ be the digit vector of $l$. If those two digit vectors differ at three or more coordinates, then $k \not\sim l$. Thus if the Hamming weight $\varkappa(k-l) \ge 3$, then $k \not\sim l$. Thus, if for all $\bsk, \bsl \in \mathcal{D}$ with $\bsk \neq \bsl$, the Hamming weight $\varkappa(\bsk-\bsl) \ge 2s+1$, then there exists a coordinate $j$ for which the Hamming weight $\varkappa(k_j-l_j)$ is at least three and therefore $r(\bsk,\bsl) = 0$. Chen and Skriganov~\cite{CS02} gave a construction of digital nets for which the NRT weight is $m+1$ and the Hamming weight is at least $2s+1$. The result then follows from \eqref{bound_L2_diagonal} and \eqref{CS_L2}. A precise result can be stated as follows.

\begin{theorem}[Chen and Skriganov, 2002]\label{csle2D}
Let $\cP_{b^m,s}$ be a digital $(t,m,s)$-net over $\mathbb{F}_b$ with a prime number $b \ge 2 s^2$ such that $$\min_{\bsk \in (\mathcal{D} \setminus \{\bszero \}) \cap \{0,1,\ldots, b^m-1\}^s} \varkappa(\bsk) \ge 2s+1.$$ Then the $\LL_2$ discrepancy of the point set $\cP_{b^m,s}$ can be bounded by $$\LL_{2,b^m}(\cP_{b^m,s}) \ll_{s,b} \frac{(m+1)^{\frac{s-1}{2}}}{b^{m-t} }.$$
\end{theorem}



\subsubsection*{The approach by Dick and Pillichshammer}

The second approach from \cite{DP13} uses a different route. Here the aim is not to avoid the cases where $\bsk \neq \bsl$, but rather to ensure that those terms do not contribute overly to the discrepancy. The NRT weight ensures that for $\bsk \in \mathcal{D}$ we have $\mu_1(\bsk)$ is large, which in turn implies that $r(\bsk,\bsk)$ is small. However, that does not ensure that the nondiagonal terms $r(\bsk,\bsl)$, where $\bsk \neq \bsl$, are small. To achieve this we demand that $\mu_2(\bsk)$ is large for all $\bsk \in \mathcal{D}$. Thus, instead of demanding large NRT weight, we now construct point sets with large $\mu_2$ weight (we note it can be shown that if $\mu_2(\bsk)$ is large, then $\mu_1(\bsk)$ has to be large too). A construction of point sets for which the dual space $\mathcal{D}$ has large weight $\mu_2(\bsk)$ will be shown below. (For technical reasons, \cite{DP13} uses $\mu_3$ for finite point sets and $\mu_5$ for sequences instead of $\mu_2$, but it is conjectured that $\mu_2$ is enough to 
achieve the optimal rate of convergence.) The following result is \cite[Theorem~2]{DP13} (note that a digital net which satisfies $\min_{\bsk \in \mathcal{D} \setminus \{\bszero\}} \mu_\alpha(\bsk) \ge \alpha m-t$ for some $\alpha \ge 3$ is a so-called order $3$ digital net).

\begin{theorem}[Dick and Pillichshammer, 2013]\label{thm_dp2}
Let $\alpha \in \mathbb{N}$ with $\alpha \ge 3$ and $\cP_{2^m,s}$ be a digital $(t,m,s)$-net over $\FF_2$ such that
\begin{equation*}
\min_{\bsk \in \mathcal{D} \setminus \{\bszero\}} \mu_\alpha(\bsk) \ge \alpha m-t.
\end{equation*}
Then the $\LL_2$ discrepancy of the point set $\cP_{2^m,s}$ can be bounded by
\begin{equation*}
\LL_{2, 2^m}(\cP_{2^m,s}) \ll_{\alpha, s} \frac{m^{\frac{s-1}{2}}}{2^{m-t}}.
\end{equation*}
\end{theorem}

\subsubsection*{The extension to arbitrary $N$}

Theorems~\ref{csle2D} and \ref{thm_dp2} are thus far only valid for special values of $N$ of the form $b^m$ ($b$ a prime greater or equal to $2s^2$ in the case of Chen and Skriganov, and $b=2$ in the case of \cite{DP13}). A method of obtaining constructions for arbitrary $N$ has been shown in \cite{CS02}.


\begin{lemma}
\label{csle2E}
Let $b \ge 2$ and $m \ge 1$ be integers. Let $\cP_{b^m,s} = \{\bsx_0, \bsx_1,\ldots, \bsx_{b^m-1}\}$ be a point set in $[0,1)^s$. Assume that the projection of $\cP_{b^m,s}$ onto the first coordinate is a $(0,m,1)$-net in base $b$, i.e. let $\bsx_n = (x_{1,n}, x_{2,n}, \ldots, x_{s,n})$ and assume that $$| \{x_{1,0}, x_{1,1}, \ldots, x_{1,b^m-1}\} \cap [0, r b^{-m}) | = r$$ for all $0 \le r < b^m$.

Then for any integer $b^{m-1} < N \le b^m$ one can construct an $N$-element point set $\cP_{N,s}$ in $[0,1)^s$ such that the $\LL_2$ discrepancy of $\cP_{N,s}$ satisfies $$N \LL_{2,N}(\cP_{N,s})  \le  \sqrt{b} b^{m} \LL_{2, b^m}(\cP_{b^m,s}).$$
\end{lemma}

We present the proof of Lemma~\ref{csle2E} in this survey since it explains the explicit construction of the point set $\cP_{N,s}$.

\begin{proof}[Proof of Lemma~\ref{csle2E}]
By the assumptions of the theorem, the point set $$\widetilde{\cP}_{N,s}:=\cP_{b^m,s} \cap
 \left(\left[0,N b^{-m}\right) \times [0,1)^{s-1}\right)$$
contains exactly $N$ points. We define the point set
$$\cP_{N,s}:=\left\{\left(N^{-1} b^{m} x_1,x_2,\ldots,x_s\right)\, :
\, (x_1,x_2,\ldots,x_s) \in \widetilde{\cP}_{N,s}\right\}.$$ Then we
have (with $\bsy=(y_1,y_2,\ldots,y_s)$)
\begin{eqnarray*}
\lefteqn{(N \LL_{2,N}(\cP_{N,s}))^2  =  \int_{[0,1]^s} \left|N \Delta_{\cP_{N,s}}(\bsy)\right|^2 \rd \bsy}\\
& = &  \int_{[0,1]^s} \left|A_N\left(\left[0,\frac{N}{b^{m}}y_1\right) \times \prod_{i=2}^s[0,y_i),\widetilde{\cP}_N\right)-b^{m} \frac{N}{b^{m}} y_1\cdots y_s \right|^2 \rd \bsy\\
& = & \frac{b^{m}}{N} \int_0^{N/b^{m}}\int_0^1 \cdots \int_0^1\left|A_N([\bszero,\bsy),\widetilde{\cP}_N)-b^{m} y_1\cdots y_s \right|^2 \rd \bsy\\
& = & \frac{b^{m}}{N} \int_0^{N/b^{m}}\int_0^1 \cdots \int_0^1\left|A_{b^{m}}([\bszero,\bsy),\cP_{b^m,s})-b^{m} y_1\cdots y_s \right|^2 \rd \bsy\\
& \le & \frac{b^{m}}{N} \left(b^{m} \LL_{2,b^{m}}(\cP_{b^m,s})\right)^2.
\end{eqnarray*}
We obtain
\begin{eqnarray*}
N \LL_{2,N}(\cP_{N,s}) \le \sqrt{b} b^{m} \LL_{2, b^m}(\cP_{b^m,s}).
\end{eqnarray*}
\end{proof}

Lemma~\ref{csle2E} can be directly applied to Theorems~\ref{csle2D} and \ref{thm_dp2}.
\begin{corollary}\label{cor_N}
Let $\cP_{b^m,s}$ be a $b^m$-element point set in $[0,1)^s$ whose projection onto the first coordinate is a $(0,m,1)$-net in base $b$ and which satisfies
\begin{equation*}
\LL_{2,b^m}(\cP_{b^m,s}) \ll_{s,b} \frac{m^{\frac{s-1}{2}}}{b^{m-t}}.
\end{equation*}
Then for any $b^{m-1} < N \le b^m$ there exists an $N$-element point set $\cP_{N,s}$ in $[0,1)^s$ such that
\begin{equation*}
\LL_{2, N}(\cP_{N,s}) \ll_{s,b} b^t\frac{(\log N)^{\frac{s-1}{2}}}{N}.
\end{equation*}
\end{corollary}

\section{The construction of finite point sets according to Chen and Skriganov}\label{sec_CS}

In this section we present Chen and Skriganov's \cite{CS02} explicit construction of finite point sets with optimal order of $\LL_2$ discrepancy. Our presentation uses the concept of digital nets rather than the original definition given in \cite{CS02}. See \cite[Corollary~16.29]{DP10} for the proof of the equivalence of these constructions.

Let $\alpha,s, m \in \mathbb{N}$ and $b$ be a prime satisfying $b \ge \alpha s$. There exist $\alpha s$ distinct elements $\beta_{i,l} \in \mathbb{F}_b$ for $i=1,\ldots, s$ and $l = 1,\ldots, \alpha$. We define $\alpha m \times \alpha m$ generating matrices $C_1,\ldots ,C_s$ over $\FF_b$ by $$C_i=(c_{u,v}^{(i)})_{u,v=1}^{\alpha m},$$  with $$c_{(l-1) m +j,k}^{(i)}={k-1 \choose j-1} \beta_{i,l}^{k-j}$$ for $j=1, \ldots , m$, $l=1, \ldots, \alpha$, $k=1, \ldots, \alpha m$ and $i=1, \ldots ,s$, where ${i \choose j}$ denotes a binomial coefficient modulo $b$ with the usual convention that ${i \choose j}=0$ whenever $j>i$ and $0^0 :=1$. Let $\cP^{(\alpha,m)}$ be the digital net over $\mathbb{F}_b$ generated by $C_1,\ldots, C_s$. Let $\mathcal{D}$ denote the dual net of $\cP^{(\alpha, m)}$. Note that for $\alpha = 1$ the generating matrices defined above were first introduced by Faure~\cite{fau82} (see also \cite[p.~92]{niesiam}).

\begin{theorem}[Chen and Skriganov, 2002]\label{lecs2Ev1}
For every prime $b$ and $\alpha, m, s \in \NN$ satisfying $b \ge \alpha s$, the digital net $\cP^{(\alpha,m)}$ is a $(0, \alpha m, s)$-net over $\mathbb{F}_b$ which satisfies $$\min_{\bsk \in (\mathcal{D}\setminus\{\bszero\}) \cap \{0, 1, \ldots, b^{\alpha m}-1\}^s} \varkappa(\bsk) \ge \alpha + 1.$$
\end{theorem}

\begin{example}\rm
For $\alpha =m = s= 2$ we may choose $b=5$, i.e., we  obtain a $(0,4,2)$-net $\cP^{(2,2)}$ over $\FF_5$. We choose different $\beta_{1,1},\beta_{1,2},\beta_{2,1},\beta_{2,2}\in \FF_5$. Then the $i$th generating matrix, $i \in \{1,2\}$ is given by $$C_i=\left(\begin{array}{cccc}
{0 \choose 0} \beta_{i,1}^0 & {1 \choose 0} \beta_{i,1}^1 & {2 \choose 0} \beta_{i,1}^2 & {3 \choose 0} \beta_{i,1}^3\\
0 &  {1 \choose 1} \beta_{i,1}^0 & {2 \choose 1} \beta_{i,1}^1 & {3 \choose 1} \beta_{i,1}^2\\
{0 \choose 0} \beta_{i,2}^0 & {1 \choose 0} \beta_{i,2}^1 & {2 \choose 0} \beta_{i,2}^2 & {3 \choose 0} \beta_{i,2}^3\\
0 &  {1 \choose 1} \beta_{i,2}^0 & {2 \choose 1} \beta_{i,2}^1 & {3 \choose 1} \beta_{i,2}^2
\end{array}\right).$$ For example, if we choose $\beta_{1,1}=0$, $\beta_{1,2}=1$, $\beta_{2,1}=2$ and $\beta_{2,2}=3$, then $$C_1=\left(\begin{array}{cccc}
1 & 0 & 0 & 0\\
0 & 1 & 0 & 0\\
1 & 1 & 1 & 1\\
0 & 1 & 2 & 3
\end{array}\right)\;\;\; \mbox{ and }\;\;\;C_2=\left(\begin{array}{cccc}
1 & 2 & 4 & 3\\
0 & 1 & 4 & 2\\
1 & 3 & 4 & 2\\
0 & 1 & 1 & 2
\end{array}\right) .$$
\end{example}

Under the assumption that $\alpha = 2s$, Theorem~\ref{lecs2Ev1} in conjunction with Corollary~\ref{cor_N} gives, for every $N,s \in \NN$ with $N \ge 2$, an explicit construction of an $N$-element point set $\cP_{N,s}$ in $[0,1)^s$ whose $\LL_2$ discrepancy is best possible with respect to the general lower bound due to Roth from Theorem~\ref{thm1}.

We close this section by remarking that a simplified version of the proofs in \cite{CS02} can be found in \cite{CS3} and in \cite{Skr}. See also \cite[Chapter~16]{DP10} for an overview.

\section{The construction of infinite sequences according to Dick and Pillichshammer}\label{sec_DP}

In this section we first present an alternate explicit construction of finite point sets with optimal $\LL_2$ discrepancy and also an explicit construction of an infinite sequence with optimal order of $\LL_2$ discrepancy in the sense of Proinov's lower bound from Theorem~\ref{thm2}. These constructions are based on digital nets and  sequences over the finite field $\FF_2$ of order $2$ (independent of the dimension $s$) in contrast to the construction of Chen and Skriganov which relied on a finite field $\FF_b$ with $b \ge 2s^2+1$.

We briefly recall a special case of generalized Niederreiter sequences as introduced by Tezuka~\cite{Tez93}. This construction is closely related to Sobol's~\cite{sob67} and Niederreiter's~\cite{nie87} construction for the generating matrices of digital sequences over $\FF_2$. We explain how to construct the entries $c_{j,k,\ell} \in \FF_2$ of the generator matrices $C_j = (c_{j,k,\ell})_{k,\ell \ge 1}$ for $j=1,2,\ldots,s$. To this end choose the polynomials $p_1=x$ and $p_j \in \FF_2[x]$ for $j =2,\ldots,s$ to be the $(j-1)$th irreducible polynomial in a list of irreducible polynomials over $\FF_2$ that is sorted in increasing order according to their degree $e_j = \deg(p_j)$, that is, $e_2 \le e_3 \le \cdots \le e_{s-1}$ (the ordering of polynomials with the same degree is irrelevant). We also put $e_1=\deg(x)=1$.

Let $j \in \{1,\ldots,s\}$ and $k \in \NN$. Take $i-1$ and $z$ to be respectively the main term and remainder when we divide $k-1$ by $e_j$, so that   $k-1  = (i-1) e_j + z$, with $0 \le z < e_j$. Now consider the Laurent series expansion
\begin{equation*}
\frac{x^{e_j-z-1}}{p_j(x)^i} = \sum_{\ell =1}^\infty a_\ell(i,j,z) x^{-\ell} \in \FF_2((x^{-1})).
\end{equation*}
For $\ell \in \mathbb{N}$ we set
\begin{equation}\label{def_sob_mat}
c_{j,k,\ell} = a_\ell(i,j,z).
\end{equation}
Every digital sequence with generating matrices $C_j = (c_{j,k,\ell})_{k,\ell \ge 1}$ for $j=1,2,\ldots,s$ found in this way is a special instance of a so-called generalized Niederreiter sequence (see \cite{Tez93} or \cite[Chapter~8]{DP10} for an overview). Note that in the construction above we always have $c_{j,k,\ell}=0$ for all $k > \ell$. Further we remark that it is well known, see for instance \cite[Theorem~4.49]{niesiam}, that Sobol's and Niederreiter's sequence are digital $(t,s)$-sequences with
\begin{equation*}\label{t_digseq}
t = \sum_{j=1}^s (e_j-1).
\end{equation*}
For generalized Niederreiter sequences this result was shown by Tezuka~\cite{Tez93}.

For the construction of the desired sequences we need the following definition.

\begin{definition}\rm
For $\alpha \in \NN$ the {\it digit interlacing composition} (with interlacing factor $\alpha$) is defined by
\begin{eqnarray*}
\mathscr{D}_\alpha: [0,1)^{\alpha} & \to & [0,1) \\
(x_1,\ldots, x_{\alpha}) &\mapsto & \sum_{a=1}^\infty \sum_{r=1}^\alpha
\xi_{r,a} 2^{-r - (a-1) \alpha},
\end{eqnarray*}
where $x_r \in [0,1)$ has dyadic expansion of the form $x_r = \xi_{r,1} 2^{-1} + \xi_{r,2} 2^{-2} + \cdots$ for $r=1, \ldots , \alpha$. We also define this function for vectors by setting
\begin{eqnarray*}
\mathscr{D}_\alpha^s: [0,1)^{\alpha s} & \to & [0,1)^s \\
(x_1,\ldots, x_{\alpha s}) &\mapsto & (\mathscr{D}_\alpha(x_1,\ldots, x_\alpha),  \ldots, \mathscr{D}_\alpha(x_{(s-1) \alpha +1},\ldots, x_{\alpha s})),
\end{eqnarray*}
for point sets $\cP_{N,\alpha s} = \{\bsx_0,\bsx_1, \ldots, \bsx_{N-1}\} \subseteq [0,1)^{\alpha s}$ by setting
\begin{equation*}
\mathscr{D}_\alpha^s(\cP_{N,\alpha s}) = \{\mathscr{D}_\alpha^s(\bsx_0), \mathscr{D}_\alpha^s(\bsx_1), \ldots, \mathscr{D}_\alpha^s(\bsx_{N-1})\}\subseteq[0,1)^s
\end{equation*}
and for sequences $\cS_{\alpha s} = (\bsx_0, \bsx_1, \ldots)$ with $\bsx_n \in [0,1)^{\alpha s}$ by setting
\begin{equation*}
\mathscr{D}_\alpha^s(\cS_{\alpha s}) = (\mathscr{D}_\alpha^s(\bsx_0), \mathscr{D}_\alpha^s(\bsx_1), \ldots).
\end{equation*}
\end{definition}

We comment here that if $\cP_{N,\alpha s}$ and $\cS_{\alpha s}$ are digital nets and sequences, respectively, then the interlacing can also be applied to the generating matrices $C_1,\ldots, C_{\alpha s}$ directly as described in \cite[Section~4.4]{D08}. This is done in the following way. Let $C_1, \ldots, C_{\alpha s}$ be generating matrices of a digital net or sequence and let $\vec{c}_{j,k}$ denote the $k$th row of $C_j$. Then the  matrices $E_1,\ldots, E_s$, where the $k$th row $\vec{e}_{j,k}$ of $E_j$ is given by
\begin{equation*}
\vec{e}_{j,u\alpha + v} = \vec{c}_{(j-1) \alpha + v, u+1}
\end{equation*}
for all $j=1, \ldots,s$ and $v=1, \ldots , \alpha$, and $u \ge 0$ are the generating matrices of $\mathscr{D}_\alpha^s(\cP_{N,\alpha s})$ or $\mathscr{D}_\alpha^s(\cS_{\alpha s})$ respectively. In particular, $\mathscr{D}_\alpha^s(\cP_{N,\alpha s})$ is an $s$-dimensional digital net and $\mathscr{D}_\alpha^s(\cS_{\alpha s})$ is an $s$-dimensional digital sequence.

Above we assumed that $c_{j,k,\ell}=0$ for all $k > K(\ell)$. Let $E_j = (e_{j,k,\ell})_{k, \ell \in \mathbb{N}}$. Then the interlacing construction yields that $e_{j,k,\ell} = 0$ for all $k > \alpha K(\ell)$, where $\alpha$ is the interlacing factor.

Let $\cS_{\alpha s}$ be a digital $(t,\alpha s)$-sequence. It follows from \cite[Definition~2 and Proposition~1]{DP13} that the first $2^m$ points of $\mathscr{D}_\alpha^s(\cS_{\alpha s})$ satisfy
\begin{equation*}
\min_{\bsk \in \mathcal{D} \setminus \{\bszero\}} \mu_\alpha(\bsk) \ge \alpha m - t_\alpha,
\end{equation*}
where
\begin{equation*}
t_\alpha = \alpha t + s {\alpha \choose 2}.
\end{equation*}
Thus Theorem~\ref{thm_dp2} implies the following result.
\begin{theorem}\label{thm4}
Let $\alpha, m, s \in \NN$ with $\alpha \ge 3$ and let $\cP_{2^m,\alpha s}$ be a digital $(t,m,\alpha s)$-net over $\FF_2$. Then the $\LL_2$ discrepancy of the digital net $\cP^{(\alpha)}_{2^m,s} = \mathscr{D}_\alpha^s(\cP_{2^m,\alpha s})$ satisfies $$\LL_{2,2^m}(\cP^{(\alpha)}_{2^m,s}) \ll_{\alpha, s, t} \frac{m^{\frac{s-1}{2}}}{2^m}.$$
\end{theorem}
The proof of this result can be found in \cite[Section~ 4]{DP13}. 



To construct finite point sets for any integer $N \ge 2$ we proceed in the following way. Let $m \in \NN$ be such that $2^{m-1} <N \le 2^m$ and let $\bsx_0, \bsx_1,\ldots, \bsx_{2^m-1} \in [0,1]^{3s-1}$ be the first $2^m$ points from the Sobol' or Niederreiter sequence in dimension $3s-1$ as introduced above with $p_1=x$ and $p_2= 1+ x$. Let $\bsx_n = (x_{1,n},\ldots, x_{3s-1,n})$ and define $\bsy_n = (n 2^{-m}, x_{1,n},\ldots, x_{3s-1,n}) \in [0,1)^{3s}$. Let now
\begin{equation}\label{defps2}
\cP^{(3)}_{2^m,s} = \{\mathscr{D}_3(\bsy_0), \mathscr{D}_3(\bsy_1), \ldots, \mathscr{D}_3(\bsy_{2^m})\}.
\end{equation}
It was shown in \cite{DP13} that the projection of this point set onto the first coordinate is a $(0,m,1)$-net in base 2.

Choose $m \in \NN$ such that $2^{m-1} < N \le 2^m$ and take the subset
$$\widetilde{\cP}_{N,s}:=\cP_{2^m,s} \cap
\left(\left[0,\frac{N}{2^m}\right) \times [0,1)^{s-1}\right)$$
which contains exactly $N$ points since the projection onto the first coordinate is $(0,m,1)$-net. Then define the point set
\begin{equation}\label{pspropcs}
\cP_{N,s}^{(3)} :=\left\{\left(2^m N^{-1} x_1,x_2,\ldots,x_s\right)\, :
\, (x_1,x_2,\ldots,x_s) \in \widetilde{\cP}_{N,s}\right\}.
\end{equation}
Corollary~\ref{cor_N} now implies the following result.
\begin{theorem}\label{thm5}
For every $N,s \in \NN$, $N \ge 2$, the $\LL_2$ discrepancy of the point set $\cP^{(3)}_{N,s}$ given by \eqref{pspropcs} satisfies $$\LL_{2,N}(\cP^{(3)}_{N,s}) \ll_{s,t} \frac{(\log N)^{\frac{s-1}{2}}}{N}.$$
\end{theorem}
The proof of this result can be found in \cite[Section~ 4]{DP13}.

\subsubsection*{The result for sequences}

It has been shown in \cite{DP13} that the sequence $\mathscr{D}_5^s(\cS_{5s})$, where $\cS_{5s}$ is a digital sequence over $\FF_2$ in dimension $5s$ constructed as presented above, has optimal order of $\LL_2$ discrepancy. More detailed, we have the following result.

\begin{theorem}\label{thm3}
For any $s \in \NN$ let $\widetilde{\cS}_{5s}$ be a digital sequence over $\FF_2$ in dimension $5s$ constructed as presented above and let $\cS_s:=\mathscr{D}_5^s(\widetilde{\cS}_{5s})$ be the interlaced version of this sequence in $[0,1)^s$. Then for all $N \ge 2$ we have
\begin{equation*}
\LL_{2,N}(\cS_s) \ll_s  \frac{(\log N)^{(s-1)/2}}{N} \sqrt{S(N)} \ll_s \frac{(\log N)^{s/2}}{N},
\end{equation*}
where $S(N)$ is the sum-of-digits function of $N$ in base 2 representation, i.e. if $N = 2^{m_1} + 2^{m_2} + \cdots + 2^{m_r}$ with $m_1 > m_2 > \cdots > m_r \ge 0$, then $S(N)=r$. Obviously, we have $S(N) \le 1+(\log N)/(\log 2)$ for all $N \in \NN$.
\end{theorem}

The proof of this result is an extension of the proof for digital nets  based on a Walsh series representation of the $\LL_2$ discrepancy of digital nets.

\section{Extensions to the $\LL_q$ discrepancy}\label{sec_lq}


Both approaches, the one by Chen and Skriganov~\cite{CS02} and the one from \cite{DP13}, can be extended to the $\LL_q$ discrepancy for $q > 2$. In this case one has Walsh coefficients $r(\bsk_1, \bsk_2,\ldots, \bsk_q)$ in \eqref{eq_L2_dual} where $\bsk_1, \bsk_2, \ldots, \bsk_q \in \mathcal{D}$. The strategies from Chen and Skriganov~\cite{CS02} and \cite{DP13} can be extended to also obtain optimal bounds in this situation. 

The first method by Chen and Skriganov~\cite{CS02} has been extended by Skriganov~\cite{Skr}.  In order to obtain the quasi-orthogonality in this case, one needs that the Hamming weight is at least $q s+1$ to ensure that the nondiagonal terms are $0$. Since again $r(\bsk_1, \ldots, \bsk_q) = \prod_{j=1}^s r(k_{j,1}, \ldots, k_{j,s})$, we obtain that there is a coordinate for which the weight is at least $q+1$, which then guarantees that $r(k_{j,1}, \ldots, k_{j,s}) = 0$. The result which one obtains in this case is the following: for each $N, s \in \mathbb{N}$ and even integer $q$ there exists a point set $\cP_{N,s,q} \subset [0,1)^s$ consisting of $N$ points such that
\begin{equation}\label{result_Skr}
\LL_{q,N}(\cP_{N,s,q}) \ll_{s,q} \frac{(\log N)^{\frac{s-1}{2}}}{N},
\end{equation}
where the implied constant does not depend on $N$. We point out that Skriganov's construction~\cite{Skr} (which is the same as in \cite{CS02} albeit with different parameters) the prime base $b$ of the digital net needs to satisfy $b > qs^2 $. Thus the point set changes for each $q$ and one cannot directly consider the case $q \to \infty$.

The approach from \cite{DP13} has been extended in \cite{D13}. In this case, the metric $\mu_2$ is sufficient to obtain the required bound for the nondiagonal terms. The result is similar to \cite{Skr}, however, in this case the point set does not depend on $q$:  for each $N, s \in \mathbb{N}$ there exists a point set $\cP_{N,s} \subset [0,1)^s$ consisting of $N$ points such that
\begin{equation}\label{result_D13}
\LL_{q,N}(\cP_{N,s}) \ll_{s,q}  \frac{(\log N)^{\frac{s-1}{2}}}{N} \quad \mbox{for all } 2 \le q < \infty,
\end{equation}
where the implied constants do not depend on $N$.

The proofs of \eqref{result_Skr} and \eqref{result_D13} use the Littlewood-Paley inequality for the Walsh function system. The complexity of the analysis greatly increases when studying the $\LL_q$ discrepancy for even integers $q > 2$ compared to the $\LL_2$ discrepancy, thus a direct approach to obtain bounds on the $\LL_q$ discrepancy for $q > 2$ seems difficult. The essential technical tool for the analysis is \cite[Lemma~4.2]{Skr}. This result was also used in \cite{D13} to prove the bound on the $\LL_q$ discrepancy.

The paper \cite{D13} also discusses upper bounds on the $\LL_q$ discrepancy for the digital sequences over $\mathbb{F}_2$ discussed above. Therein it is shown that the explicit construction of the digital sequence $\mathcal{S}_s$ in $[0,1)^{s}$, which was introduced in Theorem~\ref{thm3}, has $\LL_q$ discrepancy bounded by
\begin{equation}\label{Lq_seq}
\mathcal{L}_{q,N}(\cS_s) \ll_{s,q} \frac{r^{\frac{3}{2}- \frac{1}{q} } }{N} \sqrt{ \sum_{v=1}^r m_v^{s-1} }
\end{equation}
for all $N = 2^{m_1} + 2^{m_2} + \cdots + 2^{m_r} \ge 2$ and $2 \le q < \infty$, where the implied constant is independent of $N$. If $N$ has a bounded number of nonzero digits (for instance $N = 2^m$), then \eqref{Lq_seq} yields
\begin{equation*}
\mathcal{L}_{q,N}(\cS_s) \ll_{s,q} \frac{(\log N)^{\frac{s-1}{2}}}{N},
\end{equation*}
which is optimal by the lower bound of Roth~\cite{Roth}. However, in general, the bound does not match Proinov's lower bound \cite{pro86}. In fact, for general $N \in \mathbb{N}$, $N \ge 2$, we only have
\begin{equation*}
\mathcal{L}_{q,N}(\cS_s) \ll_{s,q} \frac{(\log N)^{\frac{s+3}{2} - \frac{1}{q} }}{N},
\end{equation*}
which is the best one can get from \eqref{Lq_seq} for $N = 1 + 2 + 2^2 + \cdots + 2^m =  2^{m+1}-1$. It is however suggested that \eqref{Lq_seq} can be improved to match Proinov's lower bound for arbitrary $N$.

\section{Extensions to Orlicz norms of the discrepancy function}\label{sec_orlicz}

Exponential Orlicz norms of the discrepancy function were studied by Bilyk et al.~\cite{BLIV} in dimension $2$ and by Skriganov~\cite{Skr12} and Amirkhanyan, Bilyk and Lacey~\cite{ABL} in arbitrary dimension. An equivalent definition of the exponential Orlicz norm is
\begin{equation}\label{Orlicz_eq}
\|f\|_{\exp(\LL^\alpha)} \simeq \sup_{q > 1} q^{-1/\alpha} \|f\|_{\LL_q}.
\end{equation}
Thus studying the exponential Orlicz norm is equivalent to studying the dependence of the constant appearing in the $\LL_q$ discrepancy bounds on $q$. In Bilyk et al.~\cite{BLIV} matching upper and lower bounds have been obtained in dimension $s=2$ which are of the following form: For all point sets $\cP_{N,2} \subset [0,1)^2$ consisting of $N$ points we have
\begin{equation*}
\|\Delta_{\cP_{N,s}}\|_{\exp(\LL^\alpha)} \gg \frac{(\log N)^{1-1/\alpha}}{N} \ \ \mbox{ for }\ \  2 \le \alpha < \infty
\end{equation*}
and for all $m \in \mathbb{N}$ there exists a point set $\cP_{2^m,2} \subset [0,1)^2$ consisting of $2^m$ points, such that
\begin{equation*}
\|\Delta_{\cP_{N,s}}\|_{\exp(\LL^\alpha)} \ll \frac{m^{1-1/\alpha}}{2^m} \ \ \mbox{ for } \ \  2 \le \alpha < \infty.
\end{equation*}

In \cite{ABL} and \cite{Skr12} exponential Orlicz norms in arbitrary dimensions $s \in \NN$ were studied (in \cite{Skr12} the author studied the dependence of the constant on $q$, which is equivalent to the Orlicz norm via \eqref{Orlicz_eq}). The authors considered randomly shifted digital nets in base $2$ (based for instance on Sobol' sequences). The result from both papers states that there exists a digitally shifted digital net $\cP_{N,s}$ such that
\begin{equation}\label{orlicz1}
\|\Delta_{\cP_{N,s}}\|_{\exp(\LL^{\frac{2}{s + 1}})} \ll_s \frac{m^{\frac{s-1}{2}}}{2^m}.
\end{equation}
(Note that a digital shift satisfying this bound is not known, so the construction is not explicit.) In both papers it is suggested that the result can be improved to
\begin{equation}\label{orlicz2}
\|\Delta_{\cP_{N,s}}\|_{\exp(\LL^{\frac{2}{s -1 1}})} \ll_s \frac{m^{\frac{s-1}{2}}}{2^m}.
\end{equation}
To put this result into context, we note that \eqref{orlicz1} is consistent with the star-discrepancy estimate
\begin{equation*}
\LL_\infty(\cP_{N,s}) \ll_s \frac{(\log N)^s}{N},
\end{equation*}
which is weaker than well known star-discrepancy estimates for digital nets of the form $N^{-1} (\log N)^{s-1}$ (see for instance \cite[Chapter~4]{niesiam}). The conjecture \eqref{orlicz2} on the other hand is consistent with the star-discrepancy estimate $$\LL_\infty(\cP_{N,s}) \ll_s \frac{(\log N)^{s-1}}{N}.$$ However such star-discrepancy estimates are well known and hold for any digitally shifted digital net with suitable $t$ parameter, see \cite[Chapter~4]{niesiam}.

\end{document}